\documentclass[10pt, twoside]{article}
\usepackage[margin=2.54cm]{geometry}
\usepackage{authblk,amsmath,amsthm,amssymb,amsfonts}
\usepackage{graphicx,color,booktabs}
\usepackage{caption,enumerate}
\usepackage{mathtools}
\usepackage{mathrsfs}
\usepackage{float}
\usepackage{indentfirst}
\usepackage{setspace}
\usepackage{calc}
\usepackage{tikz,xcolor}

\usepackage{fourier}

\usepackage[colorlinks=true,citecolor=blue,urlcolor=blue]{hyperref}

\newcommand{\CC}{\ensuremath{\mathbb{C}}}
\newcommand{\RR}{\ensuremath{\mathbb{R}}}
\newcommand{\ZZ}{\ensuremath{\mathbb{Z}}}

\newcommand{\DD}{\ensuremath{\mathbb{D}}}
\newcommand{\md}{\mathrm{d}}
\newcommand{\me}{\mathrm{e}}
\newcommand{\ve}{\varepsilon}
\newcommand{\mo}{\mathcal{O}}
\newcommand{\mi}{\mathrm{i}}
\newcommand{\sm}{\setminus}

\newtheorem{theorem}{Theorem}[section]
\newtheorem{lemma}[theorem]{Lemma}

\numberwithin{equation}{section}

\usepackage{fancyhdr}
\usepackage{titlesec}
\titleformat{\section}{\bfseries}{\thesection .}{0.5em}{}
\titleformat{\subsection}{\it}{\thesubsection .}{0.5em}{}
\titleformat{\subsubsection}{\it}{\thesubsubsection .}{0.5em}{}
\titlespacing{\section}{0pt}{3ex plus 1ex minus .2ex}{3ex plus .2ex}
\titlespacing{\subsection}{0pt}{3ex plus 1ex minus .2ex}{3ex plus .2ex}
\titlespacing{\subsubsection}{0pt}{3ex plus 1ex minus .2ex}{3ex plus .2ex}

\allowdisplaybreaks

\definecolor{lime}{HTML}{A6CE39}
\DeclareRobustCommand{\orcidicon}{%
	\begin{tikzpicture}
		\draw[lime, fill=lime] (0,0) 
		circle [radius=0.14] 
		node[white] {{\fontfamily{qag}\selectfont \tiny ID}};	\draw[white, fill=white] (-0.0625,0.095) 
		circle [radius=0.007];
	\end{tikzpicture}
	\hspace{-2mm}}
\foreach \x in {A, ..., Z}{%
	\expandafter\xdef\csname orcid\x\endcsname{\noexpand\href{https://orcid.org/\csname orcidauthor\x\endcsname}{\noexpand\orcidicon}}
}

\begin{document}

\title{\bf\Large Complete asymptotic expansions of the Humbert function $\Psi_1$\\ for two large arguments}

\author[1]{Peng-Cheng Hang\orcidA{}\,}
\author[1]{Liangjian Hu\thanks{Corresponding author.\\
		\textit{E-mail addresses:} 
		\href{mailto:mathroc618@outlook.com}{mathroc618@outlook.com} (P.-C. Hang),
		\href{mailto:Ljhu@dhu.edu.cn}{Ljhu@dhu.edu.cn} (L. Hu),
		\href{mailto:mathwinnie@dhu.edu.cn}{mathwinnie@dhu.edu.cn} (M.-J. Luo).
		                      }
	      }
\author[1]{Min-Jie Luo\orcidB{}\,}

\affil[1]{{\normalsize School of Mathematics and Statistics, Donghua University,
		
		Shanghai 201620, People's Republic of China}
	     }

\date{}

\maketitle

\begin{abstract}
	In our recent work [SIGMA \textbf{20}, 074 (2024)], the leading behaviour of the Humbert function $\Psi_1[a,b;c,c';x,y]$ when $x\to\infty$ and $y\to +\infty$ has been derived in a direct and simple manner. In this paper, we obtain the complete asymptotics of $\Psi_1$ in the general case $x,y\to\infty$ along a new path. Indeed, our proof is based on a sharp estimate on ${}_2F_2[a,b-n;c,d-n;z]$, which is valid uniformly for $n\in\ZZ_{\geqslant 0}$ and large $z$.
	\vspace{4mm}
	
	\vspace{2mm}
	
\end{abstract}

\section{Introduction}
The confluent hypergeometric function of two variables $\Psi_1$, introduced by Humbert \cite[p. 75]{Humbert 1922}, is defined by
\[\Psi_1[a,b;c,c';x,y]:=\sum_{m,n=0}^{\infty}\frac{(a)_{m+n}(b)_m}{(c)_m(c')_n}\frac{x^m}{m!}\frac{y^n}{n!},\qquad |x|<1,|y|<\infty,\]
where $a,b\in\CC$ and $c,c'\notin\ZZ_{\leqslant 0}$. It is known that $\Psi_1$ has an extension to the region (see \cite{Hang_Luo 2024, Hang_Luo 2025})
\begin{equation}\label{Psi_1 domain}
	\DD_{\Psi_1}:=\bigl\{(x,y)\in\CC^2\colon x\ne 1,\, \left|\arg(1-x)\right|<\pi,\, |y|<\infty\bigr\}.
\end{equation}

By using inverse Laplace transformations of the Humbert functions $\Phi_2$, $\Phi_3$ and $\Xi_2$, Wald and Henkel \cite{Wald_Henkel 2018} derived the leading-order behaviour of these functions when two variables become simultaneously large. But they failed to interpret the integral for $\Psi_1$ \cite[Equation (2.4b)]{Wald_Henkel 2018}
\begin{equation}\label{Psi_1 Laplace integral}
	\Psi_1[a,b;c,c';x,y]=\frac{1}{\Gamma(a)}\int_0^{\infty}\me^{-u}u^{a-1}\,_1F_1\left[\begin{matrix}
		b \\
		c
	\end{matrix};xu\right]{}_0F_1\left[\begin{matrix}
		-\\
		c'
	\end{matrix};yu\right]\md u
\end{equation}
as a convolution and so could not establish the asymptotics of $\Psi_1$.

Recently, we \cite{Hang_Luo 2024} partially answered Wald and Henkel's problem by deriving the leading behaviour of $\Psi_1$ under the condition
\begin{equation}\label{Psi_1 y>0 condition}
	x\to\infty,y\to +\infty,\quad \left|\arg(1-x)\right|<\pi,\quad 0<\gamma_1\leqslant\frac{y}{\left|1-x\right|}\leqslant \gamma_2<\infty.
\end{equation}
Our starting point is the following expression for $\Psi_1$.

\begin{theorem}[{\cite[Theorem 3.1]{Hang_Luo 2024}}]\label{Theorem: Psi_1 series representation}
	Assume that $a,b\in\CC,\,c,c'\in\CC\sm\ZZ_{\leqslant 0}$ and $a-b,a-c\in\CC\sm\ZZ$. Then
	\begin{equation}\label{Psi_1 series--|1-x|>1}
		\Psi_1[a,b;c,c';x,y]=\frac{\Gamma(c)\Gamma(b-a)}{\Gamma(b)\Gamma(c-a)}\left(1-x\right)^{-a}V_1(x,y)+\frac{\Gamma(c)\Gamma(a-b)}{\Gamma(a)\Gamma(c-b)}\left(1-x\right)^{-b}V_2(x,y)
	\end{equation}
	holds for $|{\arg}(1-x)|<\pi$, $|x-1|>1$ and $|y|<\infty$, where
	\begin{align*}
		&V_1(x,y):=\sum_{n=0}^{\infty}\frac{\left(a\right)_n\left(c-b\right)_n}{\left(a-b+1\right)_n}
		\,_2F_2\left[\begin{matrix}
			a-c+1,a+n \\
			c',a-b+1+n
		\end{matrix};\frac{y}{1-x}\right]\frac{\left(1-x\right)^{-n}}{n!},\\
		&V_2(x,y):=\sum_{n=0}^{\infty}\frac{\left(b\right)_n\left(c-a\right)_n}{\left(b-a+1\right)_n}
		\,_2F_2\left[\begin{matrix}
			a-c+1,a-b-n\\
			c',a-c+1-n
		\end{matrix};y\right]\frac{\left(1-x\right)^{-n}}{n!}.
	\end{align*}
\end{theorem}

In this paper, we shall give detailed asymptotic analyses of $\Psi_1$ when $x\to\infty$ and $y\to\infty$. Our first result provides the full asymptotic expansion of $\Psi_1$ under the condition \eqref{Psi_1 y>0 condition}, which covers our previous result \cite[Theorem 3.6]{Hang_Luo 2024}.

\begin{theorem}\label{Theorem: Psi_1 asymptotics for y positive}
	Assume that $a,b\in\CC,\,c,c'\in\CC\sm\ZZ_{\leqslant 0}$ and $a-b,a-c\in\CC\sm\ZZ$. Then, under the condition \eqref{Psi_1 y>0 condition},
	\[\Psi_1[a,b;c,c';x,y]\sim \frac{\Gamma(c)\Gamma(c')}{\Gamma(a)\Gamma(c-b)}y^{a-2b-c'}\me^y\sum_{k=0}^{\infty}a_k(x,y)y^{-k},\]
	where $a_0(x,y)=\bigl(\frac{y}{1-x}\bigr)^b$ and in general, for any  $k\in\ZZ_{\geqslant 0}$,
	\begin{equation}\label{coefficients a_k(x,y)}
		a_k(x,y)=\sum_{j=0}^{k}\frac{\left(b\right)_j\left(b-a+c'\right)_{k-j}\left(j+b-a+1\right)_{k-j}}{j!\left(k-j\right)!}\,_3F_2\left[\begin{matrix}
			-j,j-k,c+c'-a-1\\
			b-a+c',a-b-k
		\end{matrix};1\right]\left(\frac{y}{1-x}\right)^{b+j}.
	\end{equation}
\end{theorem}

Our second result gives the full asymptotics of $\Psi_1$ for two large arguments in the remaining cases. These two results give a complete answer to Wald and Henkel's problem.

\begin{theorem}\label{Theorem: Psi_1 asymptotics in the remaining case}
	Assume that $a,b\in\CC,\,c,c'\in\CC\sm\ZZ_{\leqslant 0}$ and $a-b,a-c,b-c\in\CC\sm\ZZ$. Let $w>0$ be a number such that $w>\max\left\{\Re(a-b)+1,\Re(a-c)+2\right\}$ and that the fractional parts of $w-\Re(a-b)$ and $w-\Re(a-c)-1$ are both in the interval $(\ve,1)$, where $\ve>0$ is small.
	
	Then, under the condition
	\begin{equation}\label{Psi_1 remaining case}
		x\to\infty,y\to\infty,\quad \left|\arg(1-x)\right|<\pi,\quad \left|\arg(-y)\right|<\pi,\quad 0<\gamma_1\leqslant\left|\frac{y}{1-x}\right|\leqslant \gamma_2<\infty
	\end{equation}
	and under the restriction that $y$ is bounded away from the points $b-a+k\,(k\in\ZZ)$,
	the function $\Psi_1\equiv\Psi_1[a,b;c,c';x,y]$ admits the asymptotic expansion
	\begin{align*}
		\Psi_1={}& \frac{\Gamma(c)\Gamma(b-a)}{\Gamma(b)\Gamma(c-a)}A_1(x,y)+\frac{\Gamma(c)\Gamma(c')\Gamma(a-b)}{\Gamma(a)\Gamma(c-b)\Gamma(b-a+c')}A_2(x,y)+\frac{\Gamma(c)\Gamma(c')}{\Gamma(a)\Gamma(c-b)}A_3(x,y)\\
		& +\mo\left(\left|y\right|^{-\Re(b)-w}\right)+\mo\left(\left|y\right|^{\Re(a-2b-c')-N}\me^{\Re(y)}\right),
	\end{align*}
	where
	\begin{align*}
		A_1(x,y)={}&\sum_{k=0}^{M}\frac{\left(a\right)_k\left(c-b\right)_k}{\left(a-b+1\right)_k k!}
		\,_2F_2\left[\begin{matrix}
			a-c+1,a+k \\
			c',a-b+1+k
		\end{matrix};\frac{y}{1-x}\right]\left(1-x\right)^{-a-k},\\
		A_2(x,y)={}&\sum_{k=0}^{M}\frac{\left(a-b\right)_k\left(a-b-c'+1\right)_k}{k!}\,_2F_2\left[\begin{matrix}
			b,b-c+1-k\\
			b-a+1-k,b-a+c'-k
		\end{matrix};\frac{y}{1-x}\right]\left(\frac{y}{x-1}\right)^b\left(-y\right)^{-a-k},\\
		A_3(x,y)={}&y^{a-2b-c'}\me^y\sum_{k=0}^{N-1}a_k(x,y)y^{-k},
	\end{align*}
	with $M=\lfloor w+\Re(b-a)\rfloor\geqslant 1$, $N$ being any positive integer and $a_k(x,y)$ given by \eqref{coefficients a_k(x,y)}.
\end{theorem}

The paper is organized as follows. In Section \ref{Section: 2}, we demonstrate some estimates of the generalized hypergeometric functions ${}_2F_2[a,b\pm n;c,d\pm n;z]$ when $z$ goes all the way to infinity, which are valid uniformly for $n\in\ZZ_{\geqslant 0}$. In Section \ref{Section: 3}, we use the tools in Section \ref{Section: 2} to prove our main results. The paper concludes with further remarks in Section \ref{Section: 4}.

\textbf{Notation.} The number $C$ generically denotes a positive constant independent of the summation index $n$ and the variable $z$. By $f(n,z)=\mo(a_n g(z))\,(z\in\Omega)$, we mean that there exists a constant $K>0$ independent of $n$ and $z$ such that
\[|f(n,z)|\leqslant K|a_n||g(z)|,\quad n\in\ZZ_{\geqslant 0},\,z\in\Omega.\]
Moreover, the generalized hypergeometric function $_pF_q$ is defined by \cite[Equation (16.2.1)]{NIST-Handbook}
\begin{equation}\label{pFq definition}
	{}_pF_q\left[\begin{matrix}
		a_1,\cdots,a_p\\
		b_1,\cdots,b_q
	\end{matrix};z\right]
	\equiv
	{}_pF_q[
	a_1,\cdots,a_p;
	b_1,\cdots,b_q;z]
	:=\sum_{n=0}^{\infty}\frac{(a_1)_n\cdots(a_p)_n}{(b_1)_n\cdots(b_q)_n}\frac{z^n}{n!},
\end{equation}
where $a_1,\cdots,a_p\in\CC$ and $b_1,\cdots,b_q\in\CC\sm\ZZ_{\leqslant 0}$. Empty products and sums are taken as $1$ and $0$, respectively.

\section{Auxiliary results}\label{Section: 2}
In this section, we prove some auxiliary results, including three lemmas about explicit bounds for the ratio of two gamma functions: (i) the first is beneficial to bound the tail series; (ii) the second generalizes the simple bound of gamma ratio \cite[Equation (5.6.8)]{NIST-Handbook}; (iii) the third guarantees the uniformity of explicit bounds for ${}_2F_2[a,b\pm n;c,d\pm n;z]$ obtained in the ensuing theorems.

\begin{lemma}[{\cite[Lemma 2.1]{Hang_Luo 2024}}]\label{Lemma: bound for ratio of Pochhammer symbol}
	If $a\in\CC$ and $b\in\CC\setminus\ZZ_{\leqslant 0}$, then
	\begin{equation}\label{ratio of Pochhammer symbols}
		\left|\frac{\left(a\right)_n}{\left(b\right)_n}\right|\leqslant C\left(n+1\right)^{\Re(a-b)},\qquad n\in\ZZ_{\geqslant 0}.
	\end{equation}
\end{lemma}

\begin{lemma}\label{Lemma: bound of gamma ratio}
	If $\Re(b)>\Re(a)\geqslant 0$ and $\Re(z)>|\Im(a)|$, then
	\[\left|\frac{\Gamma(z+a)}{\Gamma(z+b)}\right|\leqslant\frac{\Gamma(\Re(b-a))}{|\Gamma(b-a)|}\me^{\frac{\pi}{2}|\Im(a-b)|}\left(|z|+\Re(a)\cos\theta+\Im(a)\sin\theta\right)^{\Re(a-b)},\]
	where $\theta=\arg(z)\in (-\frac{\pi}{2},\frac{\pi}{2})$.
\end{lemma}
\begin{proof}
	Recall the integral representation for the quotient of two gamma functions \cite[p. 33]{Paris_Kaminski 2001}: for $b>a\geqslant 0$ and $\Re(z)>0$,
	\[\frac{\Gamma(z+a)}{\Gamma(z+b)}=\frac{\me^{-\mi\theta}}{\Gamma(b-a)}\int_0^{\infty}\me^{-|z|u-au\me^{-\mi\theta}}\left(1-\me^{-u\me^{-\mi\theta}}\right)^{b-a-1}\md u.\]
	By analytic continuation, it is valid for $\Re(b)>\Re(a)\geqslant 0$ and $\Re(z)>|\Im(a)|$. Note that
	\[1-\me^{-u\me^{-\mi\theta}}=1-\me^{-u\cos\theta}\cos(u\sin\theta)-\mi\cdot \me^{-u\cos\theta}\sin(u\sin\theta)\]
	lies in the right half-plane and $|1-\me^{-\zeta}|\leqslant |\zeta|$ when $\left|\arg(\zeta)\right|\leqslant\frac{\pi}{2}$. Hence
	\begin{align*}
		\left|\frac{\Gamma(z+a)}{\Gamma(z+b)}\right|&\leqslant \frac{1}{|\Gamma(b-a)|}\me^{\frac{\pi}{2}|\Im(a-b)|}\int_0^{\infty}\me^{-(|z|+\Re(a)\cos\theta+\Im(a)\sin\theta)u}u^{\Re(b-a)-1}\md u\\
		& =\frac{\Gamma(\Re(b-a))}{|\Gamma(b-a)|}\me^{\frac{\pi}{2}|\Im(a-b)|}\left(|z|+\Re(a)\cos\theta+\Im(a)\sin\theta\right)^{\Re(a-b)},
	\end{align*}
	where the identity 
	\[
	\int_{0}^{\infty}u^{\alpha-1}\me^{-\lambda u}\md u=\frac{\Gamma(\alpha)}{\lambda^{\alpha}}~~~(\min\{\Re(\alpha),\Re(\lambda)\}>0)
	\] 
	is used.
\end{proof}

\begin{lemma}\label{Lemma: uniform bound of Pochhammer ratio}
	If $a,b\in \CC$, then
	\[\prod_{j=1}^{n}\left|\frac{-a+j-z}{-b+j-z}\right|\leqslant C\left(n+1\right)^{2|a-b|},\quad n\in\ZZ_{\geqslant 0}\]
	holds for $z$ bounded away from the points $-b+1,\cdots,-b+n$.
\end{lemma}
\begin{proof}
	Suppose that $n\geqslant 1$ and $z\in\bigcup_{j=1}^n\big\{|z+b-j|\geqslant \ve\big\}$ for some $\ve\in\big(0,\frac{1}{4}\big)$. Denote
	\[P_n(z):=\prod_{j=1}^{n}\left|\frac{-a+j-z}{-b+j-z}\right|.\]
	
	\textbf{Case 1.} If $\Re(z+b)\in(k-\ve,k+\ve)$ for some $k\in\ZZ\cap[1,n]$, there is a point $z_0=-b+k+\ve\,\me^{\mi\beta}$ with $\Re(z)=\Re(z_0)$, such that for $j\in\ZZ\cap[1,n]$,
	\[\left|-b+j-z\right|\geqslant \left|-b+j-z_0\right|.\]
	Therefore,
	\begin{align*}
		P_n(z)& \leqslant\prod_{j=1}^{n}\left(1+\frac{|a-b|}{\left|-b+j-z_0\right|}\right)=C\prod_{\substack{1\leqslant j\leqslant n\\|j-k|\geqslant 2}}\left(1+\frac{|a-b|}{\left|j-k-\ve\,\me^{\mi\beta}\right|}\right)\\
		& \leqslant C\prod_{\substack{1\leqslant j\leqslant n\\|j-k|\geqslant 2}}\left(1+\frac{|a-b|}{|j-k|-4^{-1}}\right)\leqslant C\prod_{1\leqslant\ell\leqslant n}\left(1+\frac{|a-b|}{\ell}\right)^2\\
		& \leqslant C\cdot\exp\left(\sum_{1\leqslant \ell\leqslant n}\frac{2|a-b|}{\ell}\right)\leqslant C\left(n+1\right)^{2|a-b|},
	\end{align*}
	where the inequality $\sum_{j=1}^{n}\frac{1}{j}\leqslant 1+\log n$ is used.
	
	\textbf{Case 2.} The proof is akin to that of \textbf{Case 1}, if one of the following conditions holds: (i) $\Re(z+b)\in[k+\ve,k+1-\ve]$ for some $k\in\ZZ\cap[1,n]$, (ii) $\Re(z+b)\leqslant -\ve$, (iii) $\Re(z+b)\geqslant n+\ve$.
\end{proof}

Now we establish explicit expansions of ${}_2F_2[a,b\pm n;c,d\pm n;z]$ in different cases. The first gives expansions of ${}_2F_2[a,b-n;c,d-n;z]$ for large $z\in\CC\sm\RR_{\geqslant 0}$.

\begin{theorem}\label{Theorem: 2F2(-n;-z) expansion}
	Assume that $a,b\in\CC,\,c\in\CC\sm\ZZ_{\geqslant 0}$ and $d,a-b\in\CC\sm\ZZ$. Let $w>0$ be such that $w>\max\{\Re(a),\Re(b),\Re(d)\}$ and that the fractional parts of $w-\Re(a)$ and $w-\Re(b)$ are both in the interval $(\ve,1)$, where $\ve>0$ is a small number. Then for any $n\in\ZZ_{\geqslant 0}$,
	\begin{equation}\label{2F2 uniform asymptotic expansion for -n}
		{}_2F_2\left[\begin{matrix}
			a,b-n\\
			c,d-n
		\end{matrix};-z\right]=\frac{\Gamma(c)\Gamma(d-n)}{\Gamma(a)\Gamma(b-n)}\left\{S_n(z)+T_n(z)+R_{n,w}(z)\right\}
	\end{equation}
	as $z\to\infty$ such that $\left|\arg(z)\right|<\pi$ and $z$ is bounded way from the points $-b+k\,(k\in\ZZ)$, where
	\begin{align*}
		S_n(z)&=\sum_{k=0}^{\lfloor w-\Re(a)\rfloor}\frac{\Gamma(a+k)\Gamma(b-a-n-k)}{\Gamma(c-a-k)\Gamma(d-a-n-k)}\frac{\left(-1\right)^k}{k!}z^{-a-k},\\
		T_n(z)& =\sum_{k=0}^{\lfloor w-\Re(b)\rfloor+n}\frac{\Gamma(b-n+k)\Gamma(a-b+n-k)}{\Gamma(d-b-k)\Gamma(c-b+n-k)}\frac{\left(-1\right)^k}{k!}z^{n-b-k},
	\end{align*}
	and
	\[R_{n,w}(z)=\mo\left(\left(n+1\right)^{\max\{0,\Re(d-b)\}}\left|z\right|^{-w}\right)+\mo\left(\left(n+1\right)^{2|b-d|}\left|z\right|^{\Re(a+b-c-d)}\me^{-\Re(z)}\right).\]
\end{theorem}
\begin{proof}
	Let $T$ be a positive number such that
	\[T>\max\big\{1,|\Im(a)|,|\Im(b)|,|\Im(c)|,|\Im(d)|\big\}.\]
	Denote $\theta=\arg(z)$. For $|\theta|<\pi$ and $n\in\ZZ_{\geqslant 0}$, define
	\begin{equation}\label{error term--integral}
		R_{n,w}(z)=\frac{1}{2\pi\mi}\int_{\mathcal{C}}h_n(s)z^s\md s,\quad h_n(s)=\frac{\Gamma(a+s)\Gamma(b-n+s)}{\Gamma(c+s)\Gamma(d-n+s)}\Gamma(-s),
	\end{equation}
	where $\mathcal{C}$ is a negative-oriented loop that consists of the vertical line
	\[L^v:\quad s=-w+\mi t,\quad |t|\leqslant T\]
	and the contours $L^{\pm}$ which pass to infinity in the directions $\pm \theta_0$ $\big(0<\theta_0<\frac{\pi}{2}\big)$. Further, $\mathcal{C}$ is taken to embrace all the poles of $\Gamma(-s)$ and the points
	\begin{equation}\label{poles at RHS of L^v}
		s=-a-k\ (0\leqslant k\leqslant\lfloor w-\Re(a)\rfloor),\quad s=-b+n-k\ (0\leqslant k\leqslant \lfloor w-\Re(b)\rfloor+n),
	\end{equation}
	and $L^{\pm}$ are taken to be bounded away from the points shown in \eqref{poles at RHS of L^v}.
	
	The convergence of the integral as $|s|\to\infty$ is guaranteed by the term $\exp\{-|s|\cos\theta_0\cdot\log|s|\}$ in the behaviour of the gamma function for large $|s|$. The use of Cauchy's residue theorem implies that
	\[{}_2F_2\left[\begin{matrix}
		a,b-n\\
		c,d-n
	\end{matrix};-z\right]=\frac{\Gamma(c)\Gamma(d-n)}{\Gamma(a)\Gamma(b-n)}(S_n(z)+T_n(z)+R_{n,w}(z)).\]
	See \cite[Section 2]{Lin_Wong 2018} for details. Divide $R_{n,w}(z)$ into two parts:
	\[R_{n,w}(z)=\frac{1}{2\pi\mi}\int_{L^v}+\frac{1}{2\pi\mi}\int_{L^{\pm}}=:R_v(z)+R_{\pm}(z).\]
	It remains to estimate the integrals $R_v(z)$ and $R_{\pm}(z)$, respectively.
	
	\textbf{Step 1.} \textit{Estimate of $R_v(z)$}. One may claim that for $|t|\leqslant T$,
	\begin{equation}\label{Gamma estimates in vertical line}
		\left|\frac{\Gamma(a-w+\mi t)}{\Gamma(c-w+\mi t)}\Gamma(-w-\mi t)\right|=\mo(1).
	\end{equation}
	Moreover,
	\begin{equation}\label{estimates of sine ratio in vertical line}
		\frac{\sin(d-w+\mi t)}{\sin(b-w+\mi t)}=\frac{\me^{\mi(d-w)-t}-\me^{-\mi(d-w)+t}}{\me^{\mi(b-w)-t}-\me^{-\mi(b-w)+t}}=\mo(1).
	\end{equation}
	
	If $|t|\leqslant \max\{1,|\Im(d)|\}$, apply \cite[Equations (5.6.6) and (5.6.7)]{NIST-Handbook} and \eqref{ratio of Pochhammer symbols} to yield
	\[g_n(t):=\left|\frac{\Gamma(1-d+w+n-\mi t)}{\Gamma(1-b+w+n-\mi t)}\right|\leqslant C\left(n+1\right)^{\Re(d-b)},\quad n\geqslant 0.\]
	If $|t|>\max\{1,|\Im(d)|\}$, take $r>0$ such that $w>r>\Re(d)$. Set $z_1=(1+w+n-r)-\mi(t+\Im(d))$ and then
	\[g_n(t)=\left|\frac{\Gamma(z_1+r-\Re(d))}{\Gamma(z_1+r+\mi\Im(d)-b)}\right|=:\left|\frac{\Gamma(z_1+a_1)}{\Gamma(z_1+b_1)}\right|.\]
	When $\Re(b_1-a_1)=\Re(d-b)\geqslant 1$, Lemma \ref{Lemma: bound of gamma ratio} shows that
	\[g_n(t)\leqslant \frac{\Gamma(\Re(d-b))}{\left|\Gamma(d-b)\right|}\me^{\frac{\pi}{2}|\Im(d-b)|}\left(|z_1|+(r-\Re(d))\cos\theta\right)^{\Re(b-d)}\leqslant C\left|z_1\right|^{\Re(b-d)}.\]
	When $\Re(b_1-a_1)<1$, let $A=\lceil \Re(a_1-b_1)\rceil$ and $a_1'=a_1-A-1$. Then $\Re(b_1-a_1')\geqslant 1$ and further,
	\begin{align*}
		g_n(t)& =\left|\frac{\Gamma(z_1+a_1')}{\Gamma(z_1+b_1)}(z_1+a_1-A-1)(z_1+a_1-A)\cdots(z_1+a_1-1)\right|\\
		& \leqslant C\left|z_1\right|^{\Re(a_1'-b_1)}\left|(z_1+a_1-A-1)\cdots(z_1+a_1-1)\right|\\
		& \leqslant C\left|z_1\right|^{\Re(a_1-b_1-A-1)}\left|z_1\right|^{A+1}=C\left|z_1\right|^{\Re(d-b)}\\
		& =C\left(\left(1+w+n-r\right)^2+\left(t+\Im(d)\right)^2\right)^{\frac{1}{2}\Re(d-b)}\\
		& \leqslant C\cdot \max\left\{\left(n+1\right)^{\Re(d-b)},\left(|t|+1\right)^{\Re(d-b)}\right\}.
	\end{align*}
	In summary,
	\begin{equation}\label{g_n(t) in vertical line}
		g_n(t)\leqslant C\left(n+1\right)^{\max\{0,\Re(d-b)\}}\left(|t|+1\right)^{\max\{0,\Re(d-b)\}}.
	\end{equation}
	
	Combining Euler's reflection formula with the estimates \eqref{Gamma estimates in vertical line}-\eqref{g_n(t) in vertical line} shows that for $s\in L^v$,
	\begin{align*}
		|h_n(s)|\me^{-\theta t}&= \left|\frac{\Gamma(a-w+\mi t)}{\Gamma(c-w+\mi t)}\frac{\sin(d-w+\mi t)}{\sin(b-w+\mi t)}\frac{\Gamma(1-d+w+n-\mi t)}{\Gamma(1-b+w+n-\mi t)}\Gamma(-w-\mi t)\right|\me^{-\theta t}\\
		& \leqslant C\left(n+1\right)^{\max\{0,\Re(d-b)\}}\left(|t|+1\right)^{\max\{0,\Re(d-b)\}}\me^{-\theta t},
	\end{align*}
	Hence
	\[|R_v(z)|\leqslant\frac{1}{2\pi}\left|z\right|^{-w}\int_{-T}^T |h_n(s)|\me^{-\theta t}\md t\leqslant C\left(n+1\right)^{\max\{0,\Re(d-b)\}}\left|z\right|^{-w}.\]
	
	\textbf{Step 2.} \textit{Estimate of $R_{\pm}(z)$}.	Due to \cite[Lemma 2.1]{Paris_Kaminski 2001} and Euler's reflection formula, we obtain that for $s\in L^{\pm}$,
	\begin{align*}
		h_0(s)& =\frac{\Gamma(1-c-s)\Gamma(1-d-s)}{\Gamma(1-a-s)\Gamma(1-b-s)}\Gamma(-s)\cdot\frac{\sin(\pi(c+s))\sin(\pi(d+s))}{\sin(\pi(a+s))\sin(\pi(b+s))}\\
		& =\rho_0(-s)\Gamma(-s+\alpha)\cdot\frac{\sin(\pi(c+s))\sin(\pi(d+s))}{\sin(\pi(a+s))\sin(\pi(b+s))},
	\end{align*}
	where $\alpha=a+b-c-d$ and $\rho_0(s)=\mo(1)$ for $s\to\infty$ uniformly in $\left|\arg(s)\right|\leqslant \pi-\delta\,(\delta>0)$. Clearly,
	\[\left|\frac{\sin(\pi(c+s))\sin(\pi(d+s))}{\sin(\pi(a+s))\sin(\pi(b+s))}\right|=\mo(1),\quad s\in L^{\pm}.\]
	Therefore, for $s\in L^{\pm}$, there are positive constants $K_1$ and $K_2$ independent of $n$ and $z$, such that when $|s|\geqslant K_1$ then
	\[|\rho_0(-s)|\leqslant K_2,\quad \left|\frac{\sin(\pi(c+s))\sin(\pi(d+s))}{\sin(\pi(a+s))\sin(\pi(b+s))}\right|\leqslant K_2.\]
	In addition, it follows from Lemma \ref{Lemma: uniform bound of Pochhammer ratio} that for $s\in L^{\pm}$ and $n\in\ZZ_{\geqslant 0}$,
	\[\prod_{j=1}^{n}\left|\frac{-d+j-s}{-b+j-s}\right|=\mo\left(\left(n+1\right)^{2|b-d|}\right).\]
	
	In view of the fact that
	\[h_n(s)=h_0(s)\prod_{j=1}^{n}\frac{d-j+s}{b-j+s}\]
	and the estimates above, we get
	\begin{align*}
		\left|R_{\pm}(z)\right|&\leqslant C\left(n+1\right)^{2|b-d|}\left(\int_{s\in L^{\pm},|s|<K_1}|h_0(s)z^s||ds|+\int_{s\in L^{\pm},|s|\geqslant K_1}|h_0(s)z^s||ds|\right)\\
		& \leqslant C\left(n+1\right)^{2|b-d|}\left(K_3+K_2^2\int_{s\in L^{\pm},|s|\geqslant K_1}\left|\Gamma(-s+\alpha)z^s\right||ds|\right).
	\end{align*}
	According to \cite[Lemma 2.8]{Paris_Kaminski 2001}, the contours $L^{\pm}$ can be deformed to pass through the saddle point $s=-z$ and to be bounded away from the points shown in \eqref{poles at RHS of L^v}. Hence, the integral in the last line has the order of magnitude $\mo\big(|z|^{\Re(\alpha)}\me^{-\Re(z)}\big)$.
	
	The expansion now \eqref{2F2 uniform asymptotic expansion for -n} follows from the estimates of $R_v(z)$ and $R_{\pm}(z)$.
\end{proof}

The second provides expansions of ${}_2F_2[a,b+n;c,d+n;z]$ for large $z\in\CC\sm\RR_{\geqslant 0}$.

\begin{theorem}\label{Theorem: 2F2(+n;-z) uniform bound}
	Assume that $a,b\in\CC$ and $c,d\in\CC\sm\ZZ_{\leqslant 0}$. Then for $n\in\ZZ_{\geqslant\max\{0,-\Re(a),-\Re(b)\}}$,
	\[{}_2F_2\left[\begin{matrix}
		a,b+n\\
		c,d+n
	\end{matrix};-z\right]=\left(n+1\right)^{\max\{0,\Re(d-b)\}}\mo(1)+\left(n+1\right)^{2|b-d|}\mo\left(\left|z\right|^{\Re(a+b-c-d)}\me^{-\Re(z)}\right)\]
	as $z\to\infty$ such that $\left|\arg(z)\right|<\pi$ and $z$ is bounded away from the points $-b+k\,(k\in\ZZ)$.
\end{theorem}
\begin{proof}
	It is sufficient to observe that
	\[{}_2F_2\left[\begin{matrix}
		a,b+n\\
		c,d+n
	\end{matrix};-z\right]=\frac{\Gamma(c)\Gamma(d+n)}{\Gamma(a)\Gamma(b+n)}R_{-n,0}(z),\]
	where $R_{-n,0}(z)$ is defined in \eqref{error term--integral} (The contour is suitably indented in order to avoid the pole $s=0$ of $\Gamma(-s)$). The result follows by repeating the proof of Theorem \ref{Theorem: 2F2(-n;-z) expansion}.
\end{proof}

The third establishes explicit upper bounds of ${}_2F_2[a,b\pm n;c,d\pm n;z]$ for $z\to +\infty$ and refines the rough bounds in \cite[Lemma 3.3]{Hang_Luo 2024}.

\begin{theorem}\label{Theorem: 2F2(n;z) bound for z>0}
	Let $p:=\Re(a-c)+\max\{0,\Re(b-d)\}$.
	
	{\rm (i)} Assume that $a,b\in\CC,\,c\in\CC\sm\ZZ_{\leqslant 0}$ and $d\in\CC\sm\ZZ$. Then for $n\in\ZZ_{\geqslant 0}$,
	\[{}_2F_2\left[\begin{matrix}
		a,b-n\\
		c,d-n
	\end{matrix};z\right]=\mo\left(\left(n+1\right)^{2|b-d|}z^p\me^z\right),\quad z\to+\infty.\]
	
	{\rm (ii)} Assume that $a,b\in\CC$ and $c,d\in\CC\sm\ZZ_{\leqslant 0}$. Then for $n\in\ZZ_{\geqslant 0}$,
	\[{}_2F_2\left[\begin{matrix}
		a,b+n\\
		c,d+n
	\end{matrix};z\right]=\mo\left(\left(n+1\right)^{2|b-d|}z^p\me^z\right),\quad z\to+\infty.\]
\end{theorem}

\begin{proof}
	By following the proof of \cite[Lemma 3.3]{Hang_Luo 2024}, the results follow from Lemma {\rm\ref{Lemma: uniform bound of Pochhammer ratio}}.
\end{proof}

\section{Proof of the main results}\label{Section: 3}
We are now in a position to prove our main results.

\begin{proof}[\sl\textbf{Proof of Theorem} {\rm\ref{Theorem: Psi_1 asymptotics for y positive}}]
	Recall the series representation \eqref{Psi_1 series--|1-x|>1}. We have $V_1(x,y)=\mo(1)$ since $\frac{y}{1-x}$ is bounded.  Hence the main contribution of $\Psi_1$ comes from $V_2(x,y)$. Now divide $V_2(x,y)$ into two parts:
	\begin{equation}\label{V_2(x,y) two parts}
		V_2(x,y)=\sum_{n=0}^{N-1}+\sum_{n=N}^{\infty}=:V_2^{(L)}(x,y)+V_2^{(R)}(x,y),
	\end{equation}
	where $N$ is a positive integer to be determined.
	
	For $0\leqslant n\leqslant N-1$, recall the exponential expansion of ${}_2F_2$ \cite[Equation (5.8)]{Lin_Wong 2018}
	\begin{equation}\label{2F2 exponential expansion}
		{}_2F_2\left[\begin{matrix}
			a,b-n \\
			c,d-n
		\end{matrix};y\right]=\frac{\Gamma(c)\Gamma(d-n)}{\Gamma(a)\Gamma(b-n)}\me^y\sum_{k=0}^{N-1}c_{k,n} y^{a+b-c-d-k}+\mo\left(y^{\Re(a+b-c-d)-N}\me^{y}\right),
	\end{equation}
	as $y\to+\infty$, where the coefficients $c_{k,n}$ are given by \cite[Equation (12)]{Volkmer_Wood 2014}
	\begin{equation}\label{coefficients c_{n,k}}
		c_{k,n}:=\frac{\left(c+d-a-b\right)_k\left(n+1-b\right)_k}{k!}\,_3F_2\left[\begin{matrix}
			-k,c-a,d-a-n \\
			c+d-a-b,b-n-k
		\end{matrix};1\right].
	\end{equation}
	Inserting \eqref{2F2 exponential expansion} into \eqref{V_2(x,y) two parts}, we can infer that
	\[V_2^{(L)}(x,y)=\frac{\Gamma(c')}{\Gamma(a-b)}y^{a-b-c'}\me^y\sum_{n=0}^{N-1}\frac{\left(b\right)_n}{n!}\left(1-x\right)^{-n}\sum_{k=0}^{N-1}c_{k,n}^* y^{-k}+\mo\left(y^{\Re(a-b-c')-N}\me^{y}\right),\]
	where
	\[c_{k,n}^*=\frac{\left(b-a+c'\right)_k\left(n+b-a+1\right)_k}{k!}\,_3F_2\left[\begin{matrix}
		-k,-n,c+c'-a-1\\
		b-a+c',a-b-n-k
	\end{matrix};1\right].\]
	Hence
	\begin{align}
		V_2^{(L)}(x,y)&=\frac{\Gamma(c')}{\Gamma(a-b)}y^{a-b-c'}\me^y\sum_{k=0}^{N-1}y^{-k}\sum_{j=0}^{k}c_{k-j,j}^*\frac{\left(b\right)_j}{j!}\Big(\frac{y}{1-x}\Big)^j+\mo\left(y^{\Re(a-b-c')-N}\me^{y}\right) \nonumber\\
		&=\frac{\Gamma(c')}{\Gamma(a-b)}y^{a-b-c'}\me^y\sum_{k=0}^{N-1}a_k(x,y)\left(\frac{1-x}{y}\right)^ b y^{-k}+\mo\left(y^{\Re(a-b-c')-N}\me^{y}\right),\label{V_2^(L) estimate}
	\end{align}
	where $a_k(x,y)$ is given by \eqref{coefficients a_k(x,y)}.
	
	For $n\geqslant N$, a combination of \eqref{ratio of Pochhammer symbols} and Theorem \ref{Theorem: 2F2(n;z) bound for z>0} yields
	\begin{align}
		\left|V_2^{(R)}(x,y)\right|&\leqslant Cy^p\me^y\sum_{n=N}^{\infty}\left(n+1\right)^q\left|1-x\right|^{-n}\leqslant Cy^p\me^y\sum_{n=N}^{\infty}\left(n+1\right)^q\gamma_2^n y^{-n} \nonumber\\
		&=C\gamma_2^N y^{p-N}\me^y\sum_{k=0}^{\infty}\left(k+N+1\right)^q\gamma_2^k y^{-k}=\mo\left(y^{p-N}\me^y\right),\label{V_2^(R) estimate}
	\end{align}
	where $p=\Re(a-c-c'+1)+\max\{0,\Re(c-b-1)\}$ and $q=\Re(c)-2+2|b-c+1|$.
	
	For any given positive integer $M$, take $N=M+\lceil p-\Re(a-b-c')\rceil$. Since $p\geqslant \Re(a-b-c')$, one has $N\geqslant M$ and $\Re(a-b-c')-M\geqslant p-N$. Then it follows from \eqref{V_2^(L) estimate} and \eqref{V_2^(R) estimate} that
	\begin{align*}
		\left(1-x\right)^{-b}V_2(x,y)& =\frac{\Gamma(c')}{\Gamma(a-b)}y^{a-2b-c'}\me^y\sum_{k=0}^{N-1}a_k(x,y)y^{-k}+\mo\left(y^{\Re(a-2b-c')-N}\me^y\right)+\mo\left(y^{p-\Re(b)-N}\me^{y}\right)\\
		& =\frac{\Gamma(c')}{\Gamma(a-b)}y^{a-2b-c'}\me^y\sum_{k=0}^{M-1}a_k(x,y)y^{-k}+\mo\left(y^{\Re(a-2b-c')-M}\me^y\right)+\mo\left(y^{p-\Re(b)-N}\me^{y}\right)\\
		& =\frac{\Gamma(c')}{\Gamma(a-b)}y^{a-2b-c'}\me^y\sum_{k=0}^{M-1}a_k(x,y)y^{-k}+\mo\left(y^{\Re(a-2b-c')-M}\me^y\right),
	\end{align*}
	which completes the proof.
\end{proof}

\begin{proof}[\sl\textbf{Proof of Theorem} {\rm\ref{Theorem: Psi_1 asymptotics in the remaining case}}]
	Estimates of error terms here are similar to those of Theorem \ref{Theorem: Psi_1 asymptotics for y positive}, so we omit the details and just extract the expansions.
	
	Let us start with a convergent series
	\[F(x,y):=\sum_{n=0}^{\infty}\frac{\left(a_1\right)_n\left(a_2\right)_n}{\left(b_1\right)_n\left(b_2\right)_n}\,_2F_2\left[\begin{matrix}
		a,b-n \\
		c,d-n
	\end{matrix};y\right]\left(1-x\right)^{-n}.\]
	Recall the asymptotic expansion of ${}_2F_2(y)$ \cite[Equation (5.8)]{Lin_Wong 2018}: for $0\leqslant n\leqslant N$,
	\begin{equation}\label{2F2 complete expansion}
		\begin{split}
			{}_2F_2\left[\begin{matrix}
				a,b-n \\
				c,d-n
			\end{matrix};y\right]={}&\frac{\Gamma(c)\Gamma(d)}{\Gamma(a)\Gamma(b)}\frac{\left(a_3\right)_n}{\left(b_3\right)_n}\bigg\{S_n(-y)+T_n(-y)+\me^y\sum_{k=0}^{N-1}c_{k,n} y^{a+b-c-d-k}\\
			&+\mo\left(\left|y\right|^{-w}\right)+\mo\left(\left|y\right|^{\Re(a+b-c-d)-N}\me^{\Re(y)}\right)\bigg\},
		\end{split}
	\end{equation}
	where $a_3=1-b,\,b_3=1-d$, $S_n(-y)$ and $T_n(-y)$ are defined in Theorem \ref{Theorem: 2F2(-n;-z) expansion}, and the coefficients $c_{k,n}$ are given by \eqref{coefficients c_{n,k}}. For $n\geqslant N+1$, recall the expansion \eqref{2F2 uniform asymptotic expansion for -n}. Using \eqref{ratio of Pochhammer symbols} and the definition of $F(x,y)$  then gives
	\begin{equation}\label{F(x,y) expansion}
		F(x,y)=\frac{\Gamma(c)\Gamma(d)}{\Gamma(a)\Gamma(b)}\left(S_w(x,y)+T_w(x,y)+E_N(x,y)\right)+\mo\left(\left|y\right|^{-w}\right)+\mo\left(\left|y\right|^{\Re(a+b-c-d)-N}\me^{\Re(y)}\right),
	\end{equation}
	where
	\begin{align*}
		S_w(x,y)&:=\sum_{n=0}^{\infty}\frac{\left(a_1\right)_n\left(a_2\right)_n\left(a_3\right)_n}{\left(b_1\right)_n\left(b_2\right)_n\left(b_3\right)_n}S_n(-y)\left(1-x\right)^{-n},\\
		T_w(x,y)&:=\sum_{n=0}^{\infty}\frac{\left(a_1\right)_n\left(a_2\right)_n\left(a_3\right)_n}{\left(b_1\right)_n\left(b_2\right)_n\left(b_3\right)_n}T_n(-y)\left(1-x\right)^{-n},
	\end{align*}
	and
	\[E_N(x,y):=y^{a+b-c-d}\me^y\sum_{k=0}^{N-1}y^{-k}\sum_{j=0}^{k}\frac{\left(a_1\right)_j\left(a_2\right)_j\left(a_3\right)_j}{\left(b_1\right)_j\left(b_2\right)_j\left(b_3\right)_j}c_{k-j,j}\left(\frac{y}{1-x}\right)^j.\]
	
	Direct computation gives
	\begin{align*}
		S_w(x,y)={}&\frac{\Gamma(a)\Gamma(b-a)}{\Gamma(c-a)\Gamma(d-a)}\sum_{k=0}^{\lfloor w-\Re(a)\rfloor}\frac{\left(a\right)_k\left(a-c+1\right)_k\left(a-d+1\right)_k}{\left(1\right)_k\left(a-b+1\right)_k}\\
		&\qquad\cdot\left(-y\right)^{-a-k}\,_5F_4\left[
		\begin{matrix}
			a_1,a_2,a_3,1,a-d+k+1\\
			b_1,b_2,b_3,a-b+k+1
		\end{matrix};\frac{1}{1-x}
		\right]\\
		={}&\frac{\Gamma(a)\Gamma(b-a)}{\Gamma(c-a)\Gamma(d-a)}\sum_{k=0}^{\lfloor w-\Re(a)\rfloor}\frac{\left(a-d+1\right)_k}{\left(a-b+1\right)_k}s_k(x,y)\left(-y\right)^{-a-k}+\mo\left(\left|y\right|^{-\Re(a)-\lfloor w-\Re(a)\rfloor-1}\right),
	\end{align*}
	where
	\[s_k(x,y)=\sum_{j=0}^{k}\frac{\left(a_1\right)_j\left(a_2\right)_j\left(a_3\right)_j}{\left(b_1\right)_j\left(b_2\right)_j\left(b_3\right)_j}\frac{\left(a\right)_{k-j}\left(a-c+1\right)_{k-j}}{\left(1\right)_{k-j}}\left(\frac{y}{x-1}\right)^j.\]
	
	Write $K=\lfloor w-\Re(b)\rfloor$ and $\gamma=\frac{y}{1-x}$. Then
	\begin{align*}
		T_w(x,y)& =\left(-y\right)^{-b}\sum_{n=0}^{\infty}\sum_{k=0}^{n+K}\frac{\left(a_1\right)_n\left(a_2\right)_n\left(a_3\right)_n}{\left(b_1\right)_n\left(b_2\right)_n\left(b_3\right)_n}\frac{\Gamma(b-n+k)\Gamma(a-b+n-k)}{\Gamma(d-b-k)\Gamma(c-b+n-k)}\frac{\left(-1\right)^k}{k!}\left(-y\right)^{n-k}\left(1-x\right)^{-n}\\
		& =\frac{\Gamma(b)\Gamma(a-b)}{\Gamma(c-b)\Gamma(d-b)}\left(-y\right)^{-b}\sum_{n=0}^{\infty}\sum_{k=0}^{n+K}t_{n,k}\gamma^n\left(-y\right)^{-k},
	\end{align*}
	where
	\[t_{n,k}:=\frac{\left(b\right)_k\left(b-c+1\right)_k\left(b-d+1\right)_k}{\left(b-a+1\right)_k k!}\frac{\left(a_1\right)_n\left(a_2\right)_n\left(a_3\right)_n\left(a-b-k\right)_n}{\left(b_1\right)_n\left(b_2\right)_n\left(b_3\right)_n\left(1-b-k\right)_n\left(c-b-k\right)_n}.\]
	Now split $T_w(x,y)$ into two parts:
	\[T_w(x,y)=\frac{\Gamma(b)\Gamma(a-b)}{\Gamma(c-b)\Gamma(d-b)}\left(-y\right)^{-b}\left(T_w^-(x,y)+T_w^+(x,y)\right),\]
	where
	\begin{align*}
		T_w^-(x,y)& =\sum_{n=0}^{\infty}\sum_{k=0}^n t_{n,k}\gamma^n\left(-y\right)^{-k}=\sum_{k=0}^{\infty}\sum_{n=k}^{\infty} t_{n,k}\gamma^n\left(-y\right)^{-k}\\
		& =\sum_{k=0}^K\sum_{n=k}^{\infty}t_{n,k}\gamma^n\left(-y\right)^{-k}+\mo\left(\left|y\right|^{-K-1}\right)
	\end{align*}
	and
	\begin{align*}
		T_w^-(x,y)& =\sum_{n=0}^{\infty}\sum_{k=n+1}^{n+K}t_{n,k}\gamma^n\left(-y\right)^{-k}\\
		& =\sum_{n=0}^K\sum_{k=n+1}^{n+K}t_{n,k}\gamma^n\left(-y\right)^{-k}+\mo\left(\left|y\right|^{-K-1}\right)\\
		& =\sum_{k=1}^K\sum_{n=0}^{k-1}t_{n,k}\gamma^n\left(-y\right)^{-k}+\mo\left(\left|y\right|^{-K-1}\right).
	\end{align*}
	Since the inner sum in the last line is null when $k=0$, we derive
	\begin{align*}
		T_w(x,y)& =\frac{\Gamma(b)\Gamma(a-b)}{\Gamma(c-b)\Gamma(d-b)}\left(-y\right)^{-b}\sum_{k=0}^K\sum_{n=0}^{\infty}t_{n,k}\gamma^n\left(-y\right)^{-k}+\mo\left(\left|y\right|^{-\Re(b)-K-1}\right)\\
		& =\frac{\Gamma(b)\Gamma(a-b)}{\Gamma(c-b)\Gamma(d-b)}\sum_{k=0}^K t_k(x,y)\left(-y\right)^{-b-k}+\mo\left(\left|y\right|^{-\Re(b)-K-1}\right),
	\end{align*}
	where
	\[t_k(x,y)=\frac{\left(b\right)_k\left(b-c+1\right)_k\left(b-d+1\right)_k}{\left(b-a+1\right)_k k!}\,_5F_5\left[\begin{matrix}
		a_1,a_2,a_3,1,a-b-k\\
		b_1,b_2,b_3,1-b-k,c-b-k
	\end{matrix};\frac{y}{1-x}\right].\]
	
	The asymptotic expansion of $\Psi_1$ follows from \eqref{Psi_1 series--|1-x|>1} and the expansion \eqref{F(x,y) expansion}.
\end{proof}

\section{Concluding remarks}\label{Section: 4}
We have established the complete asymptotic expansions of the Humbert function $\Psi_1$ for two large arguments. Our derivation is based on the estimates of ${}_2F_2[a,b\pm n;c,d\pm n;z]$ for large $z$ (see Theorems \ref{Theorem: 2F2(-n;-z) expansion}--\ref{Theorem: 2F2(n;z) bound for z>0}), but the exponential expansions are not included. We conjecture that the more accurate expansion \eqref{2F2 complete expansion} is valid for $n\in\ZZ$ with remainders multiplied by $n^{\lambda}$ for some $\lambda>0$. Braaksma's work \cite{Braaksma 1963} seems helpful to examine this conjecture.

Regarding asymptotics of $\Psi_1$ for large $y$, it was partially studied in \cite[Section 3.1]{Hang_Luo 2025}. The series representations \eqref{Psi_1 series--|1-x|>1} and \cite[Equation (3.4)]{Hang_Luo 2024} appear to provide complete asymptotic expansions for large $y$ with unrestricted $\arg(x)$ and $\arg(y)$. We will determine appropriate methods to verify this. Additionally, the asymptotics of other Humbert functions will be explored in future studies.

Here we would like to make further remarks on our Theorems \ref{Theorem: Psi_1 asymptotics for y positive} and \ref{Theorem: Psi_1 asymptotics in the remaining case}. Note that $\Psi_1$ has a Kummer-type transformation \cite[Equation (2.54)]{Choi_Hasanov 2011}
\begin{equation}\label{Psi_1 Kummer transformation}
	\Psi_1[a,b;c,c';x,y]=\left(1-x\right)^{-a}\Psi_1\left[a,c-b;c,c';\frac{x}{x-1},\frac{y}{1-x}\right],
\end{equation}
which elucidates that the singularity of $\Psi_1[x,y]$ as $x\to 1$ is equivalent to its asymptotics under the condition \eqref{Psi_1 remaining case}. As an example, one may get from \eqref{Psi_1 Kummer transformation} and Theorem \ref{Theorem: Psi_1 asymptotics for y positive} that for fixed $y>0$,
\begin{equation}\label{example}
	\Psi_1\left[1,\frac{1}{2};\frac{3}{2},\frac{1}{2};x,y\right]\sim \frac{\sqrt{\pi}}{2}y^{-\frac{1}{2}}\left(1-x\right)^{\frac{1}{2}}\me^{\frac{y}{1-x}},\quad x\to 1^-,
\end{equation}
which is confirmed by \cite[Equation (3.4)]{Hang_Luo 2024} and was mentioned in \cite[p. 22]{Henkel 2025}.

Finally, by using \textsc{Mathematica} 12, we provide a numerical verification of Theorems \ref{Theorem: Psi_1 asymptotics for y positive} and \ref{Theorem: Psi_1 asymptotics in the remaining case}. The value of $\Psi_1=\Psi_1\bigl[1,\frac{1}{2};\frac{1}{3},\frac{1}{4};tx,ty\bigr]$ for large $t>0$ is evaluated by using \eqref{Psi_1 Laplace integral}, and the corresponding leading term is denoted by $L_{\Psi_1}$. Table \ref{Table 1} illustrates that the ratio $\frac{\Psi_1}{L_{\Psi_1}}$ approaches unity as $t\to +\infty$.
\begin{table}[htbp]
	\centering
	\caption{Numerical verification of $\Psi_1\bigl[1,\frac{1}{2};\frac{1}{3},\frac{1}{4};tx,ty\bigr]$.}\label{Table 1}
	
	\begin{tabular}{ccc|c}
		\hline
		$x$ & $y$ & $t$ & $\frac{\Psi_1}{L_{\Psi_1}}$ \\ \hline
		$-1$ & $2$ & $10$ & $0.971796$ \\
		$-1$ & $2$ & $100$ & $0.997355$ \\
		$-1$ & $2$ & $1000$ & $0.999737$ \\ \hline
		$-1$ & $-2$ & $10$ & $1.045341$ \\
		$-1$ & $-2$ & $100$ & $1.004387$ \\
		$-1$ & $-2$ & $1000$ & $1.000438$ \\
		\hline
	\end{tabular}
\end{table}

\section*{Acknowledgements} We would like to thank Malte Henkel for suggesting the inclusion of the example \eqref{example}, and Nico Temme for bringing Braaksma's work to our attention in a personal discussion with the first author.

\end{document}